\documentclass[a4paper,10pt]{amsart}

\usepackage{amssymb}

\newtheorem{theorem}{Theorem}
\newtheorem{lemma}[theorem]{Lemma}
\newtheorem{corollary}[theorem]{Corollary}
\newtheorem{proposition}[theorem]{Proposition}
\newtheorem{definition}{Definition}
\newtheorem{conjecture}[theorem]{Conjecture}

\def\R{\mathbb{R}}

\def\set#1{\left\{\, #1 \,\right\}}
\def\abs #1{\left| \,#1\, \right|}
\def\norm #1{\left\| \,#1\, \right\|}

\def\Leg{\mathcal{L}}

\def\calC{\mathcal{C}}

\def\calT{\mathcal{T}}

\newcommand\ARMA{\emph{Arch. Ration. Mech. Anal.} }

\newcommand\CMDA{\emph{Celestial Mech. Dynam. Astronom.} }
\newcommand\ETDS{\emph{Ergodic Theory Dynam. \mbox{Systems}} }
\newcommand\JDE{\emph{J. Differential Equations} }
\newcommand\NODEA{\emph{NoDEA Nonlinear Differential Equations Appl.} }
\newcommand\PAMS{\emph{Proc. Amer. Math. Soc.} }

\begin{document}

\title[Free time minimizers of the $N$-body problem]
{On the free time minimizers of the\\ Newtonian $N$-body problem}

\author{Adriana da Luz}
\address{Centro de Matem\'atica, Universidad
de la Rep\'ublica, Uruguay}
\email{adaluz@cmat.edu.uy}
\author{Ezequiel Maderna}
\address{Centro de Matem\'atica, Universidad
de la Rep\'ublica, Uruguay}
\email{emaderna@cmat.edu.uy}
\date{\today}
\maketitle

\begin{abstract}
In this paper we study the existence and
the dynamics of a very
special class of motions, which satisfy a strong
global minimization property.
More precisely, we call a free time minimizer
a curve which satisfies the least action
principle between any pair of its points
without the constraint of time for the variations.
An example of a free time minimizer defined
on an unbounded interval is a parabolic homothetic
motion by a minimal central configuration.
The existence of a large amount of free time
minimizers can be deduced from the weak KAM theorem.
In particular, for any choice of $x_0$, there should
be at least one free time minimizer $x(t)$
defined for all $t\geq0$ and satisfying $x(0)=x_0$.
We prove that such motions are completely parabolic.
Using Marchal's theorem
we deduce as a corollary that there are no entire
free time minimizers, i.e. defined on $\R$.
This means that the Ma\~n\'e set of the
Newtonian $N$-body problem is empty.
\end{abstract}

\section{Introduction and results}\label{sec:1}

Let $E$ be a finite dimensional Euclidean space,
and let $m_1,\dots,m_N > 0$ be the
masses of $N$ punctual bodies in $E$.
The Newtonian $N$-body problem consists in the
study of the dynamics of these bodies
when the law governing the motion is given
by the Newtonian potential $U:E^N\to (0,+\infty]$
\[U(x)=
\sum_{1\leq i<j\leq N}\;m_i\,m_j\;\norm{r_{ij}}^{-1}\]
where $x=(r_1,\dots,r_N)\in E^N$
is a configuration and $r_{ij}=r_i-r_j$.
This means that a curve $x:(a,b)\to E^N$,
$x(t)=(r_1(t),\dots,r_N(t))$,
such that $r_{ij}(t)\neq 0$ whenever $i\neq j$
is the position vector of a true motion
of the bodies (in a fixed inertial frame)
if an only if their components satisfy the
Newton's equations of motion
\[\ddot r_i=\sum_{j\neq i}\,m_j\,\norm{r_{ij}}^{-3}\,r_{ij}.\]

Newton's equations of motion can be easily
derived from the Hamilton's principle of
stationary action, which states that
the dynamics is determined by a
variational property of the trajectories.
More precisely,
according to Hamilton's principle,
the trajectories must be extremal curves
of the Lagrangian action, thus they
must satisfy the corresponding
Euler-Lagrange equation.
But in fact, as it is well known,
every extremal curve of the Lagrangian
action is locally minimizing,
in the sense that it must solve
the least action principle.
This viewpoint, in the study of the
dynamics of a given mechanical system,
is doubtlessly deep and fruitful.
Nevertheless, during all the last century,
a major problem in the case of
Newtonian gravitational model, prevented
the use of the direct method of
the calculus of variations
to prove the existence of particular solutions.
Namely, the problem is that the Newtonian potential
allows the existence of
curves with singularities (collisions) and
finite Lagrangian action.
A big breakthrough in this
problem was done by the discovery
essentially due to C.~Marchal,
of the fact that minimizing orbits
always avoid collisions
(assuming the obviously necessary
hypothesis $\dim E>1$).

Until now the mathematicians agree upon the fact
that we only dispose of a little
information about the dynamics of an arbitrary
trajectory of the Newtonian $N$-body problem,
except in the case $N\leq 3$.
After the pioneer works of J.~Chazy and K.~Sundman
at the beginning of the last century,
C.~Marchal, H.~Pollard and D.~Saari
(see for instance \cite{MarSaa},
\cite{Poll} and \cite{Saa})
were among the first in continuing
the systematic study
of the general case, that is to say,
without no restriction on the number of bodies nor
on the values of the masses.
A common factor in these works is the
\textsl{a priori} assumption that
the motion is well defined for all the future.
In other words, no singularity is encountered
in any future time. Now, Marchal's theorem
enable us to apply all this general theory
to minimizing solutions on unbounded intervals
which is the subject of this paper.

Recently, important advances were obtained in
the study of these trajectories in a more general
context. More precisely, the work of Barutello,
Terracini and Verzini (\cite{BTV11}, \cite{BTV11b})
on parabolic trajectories, extends the
analysis to a big class of homogeneous potentials.

\subsection{The variational setting
of the N-body problem}

In order to explain our main results,
let us recall before some usual notations.
The Lagrangian is the function
$L:TE^N\to (0,+\infty]$
\[L(x,v)=T(v)+U(x)=
\;\frac{1}{2}\;\sum_{i=1}^N\;m_i\,\norm{v_i}^2\;+\;U(x),\]
thus the Lagrangian action of an absolutely continuous
curve $\gamma:[a,b]\to E^N$ is
\[A(\gamma)=\int_a^b L(\gamma(t),\dot\gamma(t))\,dt\]
and takes values in $(0,+\infty]$.
We will denote by $\calC(x,y,\tau)$
the set of curves binding two given
configurations $x,y\in E^N$ in time $\tau>0$,
that is to say,
\[\calC(x,y,\tau)=
\set{\gamma:[a,b]\to E^N \textrm{ absolutely continuous }
\mid b-a=\tau,\,\gamma(a)=x, \gamma(b)=y},\]
and $\calC(x,y)$ will denote the set of curves
binding two configurations $x,y\in E^N$
without any restriction on time,
\[\calC(x,y)= \bigcup_{\tau>0}\calC(x,y,\tau)\;.\]
In all that follows we will consider curves
which minimize the action on these sets,
so we need to define the function
$\phi:E^N\times E^N\times (0,+\infty)\to\R$,
\[\phi(x,y,\tau)=
\inf\set{A(\gamma)\mid\gamma\in\calC(x,y,\tau)},\]
and the \textsl{critical action potential},
or the \textsl{Ma\~n\'e critical potential}
\[\phi(x,y)=
\inf\set{A(\gamma) \mid \gamma\in\calC(x,y)}=
\inf\set{\phi(x,y,\tau) \mid \tau>0}.\]
defined on $E^N\times E^N$.
It is important to say that in the first
definition, the infimum is reached
for every pair of configurations
$x,y\in E^N$. In the second one the
infimum is reached if and only if $x\neq y$.
As we will see, these facts are essentially due
to the lower semicontinuity of the Lagrangian action.

We can now introduce the object of study of this work.

\begin{definition}
A \textsl{free time minimizer}
defined on an interval $J\subset\R$
is an absolutely continuous curve
$\gamma:J\to E^N$ which satisfies
$A(\gamma\mid_{[a,b]})=\phi(\gamma(a),\gamma(b))$
for all compact subinterval $[a,b]\subset J$.
\end{definition}

There is a more or less evident
way to give an example
of a free time minimizer defined
on an unbounded interval.
we need before to define
the minimal configurations
of the problem. Recall that the
moment of inertia (about the origin)
of a given configuration $x\in E^N$ is
\[I(x)=\sum_{i=1}^N\;m_i\,\norm{r_i}^2.\]
We say that $a\in E^N$ is a (normal)
minimal configuration of the problem
when $I(a)=1$ and
$U(a)=\min\set{U(x)\mid x\in E^N,\; I(x)=1}$.
Also recall that a central configuration
is a configuration $a\in E^N$
which admits homothetic motions i.e.
of the form $x(t)=\lambda(t)\,a$.
This happens if and only if $a$ is
a critical point of $\tilde{U}= I^{1/2}U$
and $\lambda$ satisfies the Kepler equation
$\ddot\lambda\,\lambda^{2}=-U(a)\,I(a)^{-1}$.
Thus minimal configurations are
in particular central configurations.
For a given central configuration $a$,
we can choose a constant $\mu>0$
such that $x(t)= \mu\,t^{2/3}\,a$
is an homothetic motion.
We will see that such motions are free
time minimizers when the configuration
$a$ is minimal.

A less trivial way
to show the existence of free time minimizers
can be obtained using the
weak KAM theory.
It was proved by the second author
(see \cite{Mad}) that
the critical action potential is a
H\"older continuous distance function on $E^N$.
From this it is shown that the Hamilton-Jacobi
equation of the Newtonian $N$-body problem has
global critical solutions in a weak sense.
These solutions are viscosity solutions,
and to each one it can be associated
a lamination of the space of configurations
by free time minimizers.
More precisely, if the Hamiltonian of the
system is $H:T^*M\to [-\infty,+\infty)$,
then given a weak solution $u:E^N\to\R$
of the critical Hamilton-Jacobi equation
$H(x,d_xu)=0$, and any configuration
$x_0\in E^N$, there is a curve
$x:[0,+\infty)\to E^N$ which calibrates
$u$ and such that $x(0)=x_0$.
The fact that the curve is calibrating for
the weak KAM solution means that
$A(x\mid_{[0,t]})=u(x_0)-u(x(t))$ for all $t>0$.
Therefore the curve must be a free
time minimizer, since $u$ is a weak subsolution
of the critical Hamilton-Jacobi equation,
which can be expressed in terms of the action
potential saying that $u(x)-u(y)\leq \phi(x,y)$
for any pair of configurations $x,y\in E^N$.

\subsection{Main results}

In this paper we study the asymptotic
behavior of a free time minimizer, so
we will assume that its domain is an
interval $[t_0,+\infty)$; by
the previous observations, we know
that our object of study is not trivial.

More precisely,
we will prove that such kind
of motions are completely parabolic,
meaning that the velocity of each body goes to
zero as $t\to +\infty$. The origin of this
name comes from Chazy's classification of
the possible final evolution of motions defined
for all future time in the three body problem.
In fact,
we will show that free time minimizers must
have zero energy and that its moment of inertia
must grow like $I(x(t))\sim \alpha\, t^{4/3}$
for some positive constant $\alpha >0$.
From these facts we will deduce that
the motion must be completely parabolic.

On the other hand,
we must recall Marchal's theorem
(see \cite{Chen}, \cite{FerTer} and \cite{Mar})
which will be crucial for our proofs.
It asserts that,
if $\dim E \geq 2$ then the curves
that minimize the action in some
$\calC(x,y,\tau)$, cannot have collisions
in any interior time.
In particular we know that,
except for the one-dimensional case,
every free time minimizers defined in an
open interval
is a solution of Newton's equation.
Therefore, with our notation, we can say
that, if the Euclidean space $E$
has dimension at least $2$ and
$J\subset\R$ is an open interval, then
for every free time minimizer $\gamma:J\to E^N$
we have $\gamma(J)\subset\Omega$, where
$\Omega=\set{x\in E^N
\textrm{ such that }  U(x)<+\infty}$
denotes the set of configurations
without collisions.

Recently, using Marchal's theorem, the second
author has proved in \cite{Mad2}
that every free time minimizer defined on an
unbounded interval must have fixed center of mass.
This result allow us to use a theorem due to Pollard
for give a proof of our main theorem:

\begin{theorem}\label{main}
If $x:[t_0,+\infty)\to E^N$ is a free time minimizer
of the N-body problem in an Euclidean space $E$ of
dimension at least $2$, then $x$ corresponds to
a completely parabolic motion of the bodies.
\end{theorem}

The next result is a consequence of
theorem \ref{main} and again of
Marchal's theorem.
From its discovery, Marchal's theorem
was used to prove the existence of
special orbits by variational methods.
Commonly, the technique consists in
minimize the action in some special
class of curves, such as periodic curves
with topological or symmetry constraints,
and then apply the theorem
to prove that the minimizer is a true motion.
Here we will use Marchal's theorem
in the inverse way:

\begin{theorem}\label{nocompleteFTM}
If $\dim E\geq 2$
there are no entire free time minimizers
for the N-body problem in $E$,
that is to say, an entire motion $x:\R\to E^N$
is never a free time minimizer.
\end{theorem}

Another application of theorem \ref{main}
can be obtained by means of
the weak KAM theory.
It was recently established by A.~Venturelli
and the second author in \cite{MadVent} that
for any given configuration $x_0\in E^N$,
and every minimal normalized configuration $c\in E^N$,
there is a completely parabolic motion starting
at $x_0$ and asymptotic to a
parabolic homothetic motion by $c$.
As usual, a motion $x:[t_0,+\infty)\to E^N$
is said to be completely parabolic
in the future when $\lim_{t\to\infty} T(t)=0$.
This is equivalent to say that
all the velocities tend to zero when $t\to\infty$.
We will easily deduce from
theorem \ref{main} that free time minimizers
are completely parabolic.
On the other hand, as we have say,
associated to every critical solution
of the Hamilton-Jacobi equation there
is a lamination of the space of
configurations by calibrating curves
which are therefore free time minimizers
(see \cite{Mad} prop. 15).
Therefore, we obtain an alternative proof for the
abundance of completely parabolic motions:

\begin{theorem}\label{abudance}
Given N different positions
$r_1, r_2,\dots, r_N \in E$,
there exist $N$ velocities
$v_1, v_2,\dots, v_N \in E$
such that the motion determined
by these initial positions
and velocities is completely parabolic.
\end{theorem}

There is a subtle difference between the first
proof of this result given in \cite{MadVent}
and the proof given here.
Our proof uses the existence of a
weak KAM solution, and we lose the possibility
of choice for the limit shape of the bodies.
On the other hand, we gain a stronger minimization
property (the parabolic motion is not only
globally minimizing, that is, in every compact
subinterval of his domain, but also in free time).

As in \cite{FaMad}, the existence of weak KAM
solutions for the $N$-body problem
is obtained in \cite{Mad}
by a fixed point argument.
We hope that a more refined study of the
subject can give the existence of particular
weak KAM solutions, in such a way that the
limit shape of his calibrating
curves can be prescribed in advance.
These solutions would be similar
to the Busemann functions of a complete
non compact manifold.

We do not know as yet if there is a limit
configuration for a free time minimizer.
In fact, only we can say that, if a free time
minimizer has an asymptotic configuration in the
sense that the normalized
configuration $u(t)=I(x(t))^{-1/2}x(t)$ converges
to some configuration $a\in E^N$ with $I(a)=1$,
then the limit configuration $a$ must be a
central configuration such that its parabolic
homothetic motion is itself a free time minimizer.
On the other hand, this last property seems to
be the only requirement on the configuration
which is needed to define an associated critical
Busemann function. We refer the reader to the
work of G.~Contreras \cite{Cont} for a
construction of the critical
Busemann functions of an autonomous
Tonelli Lagrangian and his relationship
with the weak KAM theory.
These considerations show that the set of
configurations with such property is
playing the role of the
Aubry set at infinity.

Another interesting invariant set in the
general theory of Tonelli Lagrangians
is the Ma\~n\'e set.
We refer the reader to the original paper
of Ma\~n\'e \cite{Mane}
for the definition of a semistatic curve,
as well as to the work of G.~Contreras
and G.~Paternain \cite{ConPat}.
It is not difficult to see that the semistatic
curves in these cited works are precisely
the free time minimizers in our context
(the critical value is $c(L)=0$).
The Ma\~n\'e set is defined as the subset
of the tangent bundle $TM$ whose
elements are the velocity of some
entire semistatic curve
(the set $\Sigma(L)$ in the cited literature).
Therefore, theorem \ref{nocompleteFTM}
says that the Ma\~n\'e set of the
Newtonian $N$-body problem is empty.

The paper is organized as follows.
The next section is devoted to introduce
the main tools and notations we will use.
In particular, the lower semicontinuity
of the Lagrangian action is showed as well
as the homogeneity of the action potential.
In the third section, the existence of
free time minimizers is proved, and some of its
basic properties are discussed.
The last section begins recalling
a theorem of H.~Pollard
(\cite{Poll} theorem 5.1, p. 607)
and give the proof of theorem \ref{main}
and theorem \ref{nocompleteFTM}.

\section{Preliminaries and notations}

As usual, we will use the notation
$I(t)$ for $I(x(t))$ when the curve
$x(t)$ is understood.
In the same way we will write
$U(t)=U(x(t))$, $T(t)=I(\dot x(t))$.
Therefore, if $x(t)$ describes a motion
of the system, then the quantity
$h=T(t)-U(t)$ is the constant
total energy of the motion,
and the Lagrange-Jacobi
relation (or virial relation) can be
written $\ddot I=2U+4h$.

We will write $x\cdot y$
the mass inner product
of two configurations $x,y\in E^N$,
thus we have $I(x)=x\cdot x$.
It is easy to see that
Newton's equations admit
the synthetic expression
$\ddot x =\nabla U(x)$
if the gradient is taken
with respect to this inner product.
With the obvious identification
$TE^N\simeq E^N\times E^N$
we can write $2T(v)=v\cdot v$.
If we apply the
Cauchy-Schwarz inequality to the product
$x(t)\cdot v(t)$, where $v=\dot x$,
we get the inequality $2IT-\dot I^2\geq0$,
where equality holds if and only if
the velocities vector and the
configuration vector are collinear.
In particular, the equality holds on
an open interval of time if and only
if the curve is homothetic
on this interval.

An excellent presentation of
basic geometric constructions
for $N$-body problems
with homogeneous potentials is
the paper of A.~Chenciner \cite{Chen2},
to which we refer the reader for
other intimately related definitions
and properties.

\subsection{The Lagrangian action in polar coordinates}
\label{polaraction}

When a curve $x:[a,b]\to E^N$ avoid
the total collision, we can decompose
it as the product of a positive real function
by an unitary configuration. In other words,
we can write $x(t)=\rho(t)\,u(t)$,
with $\rho(t)>0$
and $I(u(t))=1$ for all $t\in [a,b]$.
Note that these factors are well defined as
$\rho=I(x)^{1/2}$ and $u=\rho^{-1}x$.

Therefore we have
$\dot x=\dot\rho\,u+\rho\,\dot u$.
Since $u^2=u\cdot u=I(u)$
is constant, we also
have $u\cdot \dot u=0$,
from which we deduce that
$\dot x^2=\dot x\cdot \dot x= \dot\rho^2+
\rho^2\,\dot u^2$.
If in addition we consider the homogeneity
of the Newtonian potential, we have that
$U(x)=U(\rho\,u)=\rho^{-1}U(u)$.

Thus we get the following expression
for the action of the curve $x$, which will
be useful to compare it to other paths
joining the same endpoints.

\begin{equation}
\label{polaraction-expression}
A(x)=\frac{1}{2}\int_a^b\dot\rho(s)^2\,ds+
\frac{1}{2}\int_a^b\rho(s)^2\dot u(s)^2\,ds+
\int_a^b\rho(s)^{-1}U(u(s))\,ds.
\end{equation}

Note that if the curve
$x$ is homothetic, the
second term vanishes.
Thus, in this case, the
action of $x$ can be viewed as
the Lagrangian action of $\rho$
as a curve in $\R^+$ with respect to the
Lagrangian associated to a reduced
Kepler problem (central force)
in this half-line.

\subsection{Lower semicontinuity of the Lagrangian action}

Frequently in the literature,
the curves are considered
in the Sobolev space $H^1$,
but it is not difficult to see
that for this kind of Lagrangian,
absolutely continuous curves
with finite action must have
square-integrable derivative,
in particular they are also
$1/2$-H\"older continuous.
If $x:[a,b]\to E^N$ is an absolutely
continuous curve
such that $A=A(x)<+\infty$,
then obviously we have
\[\int_a^b \abs{\dot x(s)}^2\,ds\leq 2A.\]
On the other hand, it is well know
that for any absolutely continuous
curve, the distance between its
extremities is bounded
by the integral of the norm
of the speed. Hence, given
$a\leq s<t\leq b$,
we can apply the Bunyakovsky
inequality, and we deduce that
\begin{equation}
\label{lowerboundaction}
\abs{x(t)-x(s)}\leq
\int_s^t\abs{\dot x(u)}\,du
\leq (2A)^{1/2}\;\abs{t-s}^{1/2}.
\end{equation}
Here we use the norm in
$E^N$ induced by the mass inner
product, which is denoted $\abs{x}$,
but it is clear that the H\"{o}lder
continuity not depends on the
choice of the norm since they
are all equivalent. Thus by Ascoli's
theorem we obtain
the following proposition.

\begin{proposition}
\label{subsequence}
Let $x_n:[a,b]\to E^N$
be a sequence of absolutely
continuous curves for which
there is a positive
constant $k<+\infty$ such that
$A(x_n)\leq k$ for all $n>0$.
If $x_n(t)$ converges for some
$t\in[a,b]$, then there is
a subsequence $x_{n_k}$
which converges uniformly.
\end{proposition}

In other words,
given $x,y\in E^N$,
$\tau>0$ and $k>0$,
we know that the sets
of absolutely continuous curves
\[\Sigma(x,y,\tau,k)=
\set{\gamma:[0,\tau]\to E^N
\mid \gamma(0)=x,\;\gamma(\tau)=y,
\textrm{ and } A(\gamma)\leq k}\]
are relatively compact
in $C^0([a,b],E^N)$.
As we will see,
the compactness of
such sets is equivalent
to the lower semicontinuity
of the Lagrangian action on
the subset of absolutely
continuous curves.
Moreover, we will see
that it is a consequence
of the well known Tonelli's
lemma for strictly convex and
superlinear Lagrangians
that we state below.

We recall that an autonomous
\textsl{Tonelli Lagrangian}
on a complete
Riemannian manifold $M$ is a
function $L:TM\to\R$
of class $C^2$, which is
strictly convex on
each fiber of $TM$,
and such that for each positive
constant $\alpha>0$ there is
$C_\alpha\in\R$ such that
$L(x,v)\geq
\alpha\,\norm{v}+C_\alpha$
for all $(x,v)\in TM$.
We will denote $A_L(x)$
the corresponding Lagrangian
action of an absolutely
continuous curve $x$ in $M$.

\begin{lemma}
[Tonelli's lower semicontinuity]
Let $M$ be a
Riemannian manifold,
and let $L:TM\to\R$ be a
Tonelli Lagrangian on $M$.
Suppose that $x_n:[a,b]\to M$
is a sequence of absolutely
continuous curves such that
$\sup A_L(x_n)<+\infty$. If
$x_n$ converges uniformly to
a curve $x$,
then the limit curve $x$ is
absolutely continuous, and
$A_L(x)\leq \liminf A_L(x_n)$.
\end{lemma}

There are several proofs
in the literature of the
above lemma,
see for instance the first
appendix in \cite{Mat}, where
an equivalent version is given.
We can deduce from this lemma
the following theorem,
also due to Tonelli,
which assures the existence of
absolutely continuous minimizers.

\begin{theorem}[Tonelli]
Let $M$ be a complete connected
Riemannian manifold,
and let $L:TM\to\R$ be an
autonomous
Tonelli Lagrangian on $M$.
Given $x,y\in M$ and $\tau>0$,
the action $A_L$ takes a minimum
value over the set of all absolutely
continuous curves
$\gamma:[0,\tau]\to M$
such that $\gamma(0)=x$
and $\gamma(\tau)=y$.
\end{theorem}

\begin{proof}
For each $k\in\R$,
let $\Sigma_k$ be the
set of absolutely continuous
curves $\gamma:[0,\tau]\to M$
such that $A_L(\gamma)\leq k$.
Since $M$ is connected,
the sets $\Sigma_k$ are
nonempty for sufficiently large
values of $k$.
From the superlinearity of $L$
we can deduce as before,
that each curve in $\Sigma_k$
is $1/2$-H\"{o}lder continuous,
with a H\"{o}lder constant which
only depends in $k$.
Therefore, since $M$ is complete,
we can apply Ascoli's theorem as
in proposition \ref{subsequence},
and we get that the sets
$\Sigma_k$ are relatively compact
in the $C^0$-topology.
But if we apply Tonelli's
lemma to a convergent sequence
in some $\Sigma_k$ we conclude
that the limit curve is also in
$\Sigma_k$. Thus, each
$\Sigma_k$ is actually compact
in the $C^0$-topology.

We note now that the
Lagrangian is bounded below,
since the superlinearity implies
that $A_L(\gamma)\geq \tau\,C_1$.
Thus $k_0= \inf\set{k\in\R
\mid\Sigma_k\neq\emptyset}$
is well defined. Since
$\Sigma_{k_0}=\cap_{k>k_0}\Sigma_k$
we can conclude that
$\Sigma_{k_0}\neq\emptyset$.
But it is clear that each
$\gamma\in\Sigma_{k_0}$ is a
minimizer of $A_L$ in
the required set of curves.
\end{proof}

Using Fatou's lemma we can obtain
a Tonelli's theorem which works
for the Lagrangian action of the
Newtonian $N$-body problem.
As before, first we need
to establish the lower
semicontinuity of the action.

\begin{lemma}
\label{LowerSC-NBP}
Let $x_n:[a,b]\to E^N$ be a
sequence of absolutely continuous
curves which converges
uniformly to a limit curve $x$,
and such that $\sup A(x_n)<+\infty$.
Then, $x$
is also an absolutely continuous
curve, and $A(x)\leq \liminf A(x_n)$.
\end{lemma}

\begin{proof}
Let $L_0$ be the quadratic
Lagrangian in $E^N$ given by
\[L_0(x,v)=\frac{1}{2}\abs{v}^2\]
and let $A_0$ be the associated action.
Certainly $L_0$ is a
Tonelli Lagrangian on $E^N$,
therefore Tonelli's lemma can be
applied.
Thus we have that $x$ is
absolutely continuous, and that
\[A_0(x)\leq \liminf A_0(x_n).\]
On the other hand, since $U$
is a positive measurable function,
by Fatou's lemma we have
\[\int_a^b U(x(s))\,ds=
\int_a^b \liminf U(x_n(s))\,ds\leq
\liminf \int_a^b U(x_n(s))\,ds.\]
Since
\[A(x)=A_0(x)+\int_a^b U(x(s))\,ds\]
we get $A(x)\leq \liminf A(x_n)$
what was required to be proved.
\end{proof}

An evident corollary of lemma \ref{LowerSC-NBP}
is the existence
of absolutely continuous minimizers
on each set of
curves $\calC (x,y,\tau)$.

\begin{theorem}[Tonelli's theorem for
the $N$-body problem]
\label{TonelliNBP}
Given two configurations
$x,y\in E^N$ and $\tau > 0$,
there is at least one curve
$\gamma\in\calC (x,y,\tau)$ such that
$A(\gamma)=\phi(x,y,\tau)$.
\end{theorem}

\begin{proof}
By proposition \ref{subsequence}
and lemma \ref{LowerSC-NBP}, we
already know that given $c\in\R$,
the sets of
absolutely continuous curves
\[\Sigma(x,y,\tau,c)=
\set{\gamma:[0,\tau]\to E^N
\mid \gamma(0)=x,\gamma(\tau)=y,
\;\textrm{and }A(x)\leq c}\]
are compact subsets of
$C^0([0,\tau],E^N)$.
Moreover, since $L>0$
and $E^N$ is connected,
they are empty for
$c\leq 0$, and nonempty for
sufficiently large values
of $c>0$.

We observe now that
$\phi(x,y,\tau)=
\inf\set{c>0\mid
\Sigma(x,y,\tau,c)\neq\emptyset}$.
Hence the intersection
for $c>\phi(x,y,\tau)$
of these nonempty compact
sets is also nonempty, and each
curve $\gamma$ in the
intersection satisfies
$A(\gamma)=\phi(x,y,\tau)$.
\end{proof}

\subsection{Regularity of minimizers and Marchal's theorem}

Everything what we said would
not be useful for anything,
unless we are able to show that
absolutely continuous minimizers
correspond to true motions or,
in other words, to solutions of
Newton's equation. We
will explain why this happens
briefly.

For a Tonelli Lagrangian
on a smooth manifold, it is
very well known that a $C^1$
minimizer is in fact of class
$C^2$ and satisfies the
Euler-Lagrange equations.
Therefore it suffices to show
that an absolute continuous
minimizer is of class $C^1$.
The proof of this fact,
which actually holds even for
time dependent Tonelli Lagrangians
(see \cite{Mat}) is a little
arduous, but for a mechanical
system is not it as much.
We must recall first a
classical result of the
calculus of variations,
which tell us that if
$\gamma\in\calC(x,y,\tau)$
is a critical point of the
Lagrangian action,
then in local coordinates
we can write
\[\Leg(\gamma(t),\dot\gamma(t))=
\left(\gamma(t),\,u+\int_0^t
\frac{\partial L}{\partial x}
(\gamma(s),\dot\gamma(s))\,ds\,\right)\]
where $\Leg:TM\to T^*M$ is the
Legendre transform, $u\in(\R^n)^*$,
and the equality holds for
almost every $t\in [0,\tau]$.

Note that for a mechanical system,
i.e. of the form $L(x,v)=g(v,v)+U(x)$,
where $g$ is a Riemannian metric,
and $V$ a smooth function en $M$,
the right hand of this equality
is a continuous function of $t$,
since (in local coordinates) we have
\[\frac{\partial L}{\partial x}
(\gamma(s),\dot\gamma(s))=
DU(\gamma(s))\]
for all $s\in[0,\tau]$.
Since $\Leg$ is a diffeomorphism
of class $C^1$, we conclude that
$\dot\gamma$ is actually
a continuous function.
We refer the reader to \cite{AbbonSch}
(proposition 3.1) for a more detailed
explanation and other basic properties
of the Lagrangian action.

Finally, we observe that our Lagrangian is
a smooth mechanical system in $\Omega$,
the open and dense
subset of $E^N$ where $U<+\infty$.
Since the above considerations are
of a local nature, we conclude that an
absolutely continuous minimizer whose
image is contained in
$\Omega$ must be smooth.

Tonelli's theorem results extremely useful
combined with Marchal's theorem that we
recall now. From such combination
and the above considerations
we can conclude that, except in the
collinear case ($\dim E=1$), Tonelli minimizers
are smooth in the interior of its domain.

\begin{theorem}
[Marchal \cite{Mar}, Chenciner \cite{Chen},
Ferrario-Terracini \cite{FerTer}]
Suppose $\dim E\geq 2$.
If $\gamma:[a,b]\to E^N$
is such that
$A(\gamma)=\phi(\gamma(a),\gamma(b),b-a)$,
then $\gamma(t)\in \Omega$ for all $t\in (a,b)$.
\end{theorem}

Combining Marchal's theorem with \ref{TonelliNBP}
and the above considerations we obtain the
following corollary.

\begin{corollary} If $\dim E\geq 2$, then
for every pair of configurations
$x,y\in E^N$, and every positive time $\tau > 0$,
there is at least one curve
$\gamma\in\calC (x,y,\tau)$ such that
\[A(\gamma)=\phi(x,y,\tau)\,\] and such that
\[\gamma(t)\in\Omega \textrm{ for every }
t\in (0,\tau)\,.\]

In particular the restriction of $\gamma$ to
$(0,\tau)$ satisfies Newton's equations, that is
to say, it is a true motion of the
$N$-body problem.
\end{corollary}

\subsection{Homogeneity of the critical action potential}
\label{homogeneity-section}

We shall introduce now some
properties which are consequence
of the homogeneity of the
Newtonian potential. Given a curve
$\gamma\in\calC (x,y,\tau)$
and two positive numbers
$\lambda,\mu >0$ we define the curve
$\gamma_{\lambda,\mu}\in
\calC(\lambda x,\lambda y,\mu)$
in the obvious way.
If $\gamma$ is defined for $t\in[a,b]$,
with $b-a=\tau$, then $\gamma_{\lambda,\mu}$
can be defined for
$s\in [\mu\tau^{-1}a,\mu\tau^{-1}b]$ by
\[\gamma_{\lambda,\mu}(s)=
\lambda\,\gamma(\tau\mu^{-1}s)\]

Using these curves we can deduce the
following lemma and corollaries.

\begin{lemma}
\label{homotecia.repar}
If $\mu=\lambda^{3/2}\tau$, then
$A(\gamma_{\lambda,\mu})=
\lambda^{1/2}A(\gamma)$.
\end{lemma}

\begin{proof}
A simple computation shows that the
action of $\gamma_{\lambda,\mu}$ is
\begin{eqnarray*}
A(\gamma_{\lambda,\mu})&\;=\;&
\frac{\lambda^2\tau^2}{\mu^2}\frac{1}{2}
\int_{\mu\tau^{-1}a}^{\mu\tau^{-1}b}
\abs{\dot\gamma(\tau\mu^{-1}s)}^2\,ds+
\frac{1}{\lambda}\int_{\mu\tau^{-1}a}^{\mu\tau^{-1}b}
U(\gamma(\tau\mu^{-1}s))\,ds\\
& &\\
&\;=\;&\frac{\lambda^2\tau}{\mu}
\frac{1}{2}\int_a^b
\abs{\dot\gamma(t)}^2\,dt+
\frac{\mu}{\lambda\tau}\int_a^b
U(\gamma(t))\,dt.
\end{eqnarray*}

It suffices now to make the substitution
$\mu=\lambda^{3/2}\tau$.
\end{proof}

\begin{corollary}
\label{phitau-homogeneo}
For all $x,y\in E^N$ and for every
$\tau,\lambda>0$ we have
\[\phi(\lambda x,\lambda y,\lambda^{3/2}\tau)=
\lambda^{1/2}\phi(x,y,\tau).\]
\end{corollary}

\begin{proof}
Let $\epsilon>0$ and
$\gamma\in\calC (x,y,\tau)$ such that
$A(\gamma)\leq \phi(x,y,\tau)+\epsilon$.
Setting $\mu=\lambda^{3/2}\tau$,
we write $\gamma_\lambda$
instead of $\gamma_{\lambda,\mu}$.
Thus we can apply lemma
\ref{homotecia.repar}, and we get that
\begin{eqnarray*}
A(\gamma_\lambda)& = &\lambda^{1/2}A(\gamma)\\
&\leq & \lambda^{1/2}\phi(x,y,\tau)+
\lambda^{1/2}\epsilon.
\end{eqnarray*}
Since
$\gamma_\lambda\in
\calC(\lambda x,\lambda y,\lambda^{3/2}\tau)$,
and $\epsilon>0$ is arbitrary, we deduce
that
\[\phi(\lambda x,\lambda y,\lambda^{3/2}\tau)
\leq\lambda^{1/2}\phi(x,y,\tau).\]
Therefore we also have
\begin{eqnarray*}
\phi(x,y,\tau) & = &
\phi(\lambda^{-1}\lambda x,
\lambda^{-1}\lambda y,
\lambda^{-3/2}\lambda^{3/2}\tau)\\
& \leq & \lambda^{-1/2}
\phi(\lambda x,\lambda y,\lambda^{3/2}\tau),
\end{eqnarray*}
which proves the reverse inequality.
\end{proof}

\begin{corollary}
\label{phi-homogeneo}
For all $x,y\in E^N$,
and every $\lambda>0$, we have
$\phi(\lambda x,\lambda y)=
\lambda^{1/2}\phi(x,y)$.
\end{corollary}

\begin{proof}
Take the
infimum over $\tau>0$
in the equality given by
corollary \ref{phitau-homogeneo}.
\end{proof}

\begin{corollary}\label{inv-rescalingFTM}
Given a free time minimizer
$\gamma:[a,b]\to E^N$, and $\lambda>0$,
the curve
\begin{eqnarray*}
\gamma_\lambda:
[\lambda^{3/2}a,\lambda^{3/2}b]&\to& E^N\\
t&\mapsto&
\gamma_\lambda(t)=
\lambda\gamma(\lambda^{-3/2}t)
\end{eqnarray*}
is also a free time minimizer.
\end{corollary}

\begin{proof}
If we denote $x=\gamma(a)$ and $y=\gamma(b)$, we have
\begin{eqnarray*}
A(\gamma_\lambda)&=&\lambda^{1/2}A(\gamma)\\
&=& \lambda^{1/2}\phi(x,y)\\
&=& \phi(\lambda x,\lambda y).
\end{eqnarray*}
On the other hand, it is clear that
$\gamma_\lambda\in\calC(\lambda x,\lambda y)$,
thus $\gamma_\lambda$ is a free time minimizer.
\end{proof}

\subsection{The Ma\~n\'e critical energy level}
\label{Manie-section}

In the Ma\~n\'e works,
the critical energy level
of a Tonelli Lagrangian
$L:TM\to\R$ on a connected
compact manifold $M$ is defined as
\[c(L)=\inf\set{c\in\R\mid
A_{L+c}(\gamma)\geq 0
\textrm{ for every closed curve }\gamma}.\]
It is easy to show that,
if for some $c\in\R$ there
is a closed curve $\gamma$
such that $A_{L+c}(\gamma)<0$,
then for any given pair of points
$x,y\in M$ the Lagrangian action
of $L+c$ has no lower bound in
the set of all absolutely
continuous curves from $x$ to $y$.
On the other hand, it can be proved that
the Lagrangian $L+c(L)$ admits free time
minimizers for any pair of prescribed endpoints
$x,y\in E^N$. Moreover, the energy constant
of this curves (also called semistatics)
is exactly $c(L)$, and using them and the compactness
of the manifold it can be proved the existence
of invariant measures supported in the critical
energy level, and the existence of
several compact invariant
sets with interesting dynamical properties.
We did not wish to develop here this theory
more than necessary to show the existing analogy,
even if our Lagrangian flow is not complete.
We will only show that in the Newtonian
$N$-body problem we have $c(L)=0$, and
that free time minimizers have zero energy.

For a natural mechanical system with bounded potential
energy $V:M\to\R$, it is very easy to see that
$c(L)=\sup V$. Suppose first that $c<\sup V$.
Then there is some open subset $A\subset M$
in which $c<V$. Any constant curve
$\gamma(t)=p\in A$ defined in some time interval
$[0,\tau]$ is a closed curve, and
clearly we have $A_{L+c}(\gamma)=\tau(c-V(p))<0$.
This proves that $c(L)\geq\sup V$. On the other
hand, if $c=\sup V$, then $L+c=T-V+c\geq 0$,
thus $A_{L+c}(\gamma)\geq 0$ whatever the
curve $\gamma$ is.
In our setting the potential energy is the
function $V=-U$, hence the critical
energy level is zero.

Suppose that
$\gamma:[a,b]\to M$ is an
absolutely continuous curve,
with finite action $A_L(\gamma)$,
and binding the points
$x=\gamma(a)$ and $y=\gamma(b)$
in time $\tau=b-a$.
A particular kind of variation
of the curve $\gamma$
can be obtained by reparametrization.
Consider for instance, for $\alpha>0$,
the linear reparametrization
$\gamma_\alpha:
[\alpha a,\alpha b]\to E^N$,
$\gamma_\alpha(s)=
\gamma(\alpha^{-1}s)$.
Thus we have
$\gamma_\alpha\in\calC(x,y,\alpha\tau)$
and $\gamma_1=\gamma$. The action of
$\gamma_\alpha$ is
\begin{eqnarray*}
A(\gamma_\alpha)&=&\alpha^{-2}\,
\frac{1}{2}\int_{\alpha a}^{\alpha b}
\abs{\dot\gamma(\alpha^{-1}s)}^2\,ds +
\int_{\alpha a}^{\alpha b}
U(\gamma(\alpha^{-1}s))\,ds\\
&=&\alpha^{-1}\,\frac{1}{2}
\int_a^b\abs{\dot\gamma(t)}^2\,dt +
\alpha\,\int_a^b U(\gamma(t))\,dt.
\end{eqnarray*}
Therefore
\[\frac{d}{d\alpha}A(\gamma_\alpha)
\mid_{\alpha=1}\,=\,
-\int_a^b\left(\frac{1}{2}\,\abs{\dot\gamma(t)}^2-
U(\gamma(t))\right)\,dt,\]
from which the following lemma can be deduced.

\begin{lemma}
\label{zeroenergy}
Let $I\subset\R$ be an open interval.
If $\gamma:I\to E^N$ is a free time minimizer,
then $\gamma$ is a trajectory with zero energy.
\end{lemma}

\begin{proof}
It is clear that for every
compact subinterval $[a,b]\subset I$,
we have that $\gamma\mid_{[a,b]}$
is a Tonelli minimizer, meaning that
$A(\gamma\mid_{[a,b]})=
\phi(\gamma(a),\gamma(b),b-a)$.
By Marchal's theorem we know that
$\gamma(t)\in\Omega$ for every $t\in (a,b)$.
Since $[a,b]\subset I$ is arbitrary, we
conclude that $\gamma(I)\subset\Omega$.
This implies that $\gamma$ is a smooth curve
which corresponds to a true motion. Therefore
$\gamma$ must have constant energy $h=T(t)-U(t)$.

We fix now some interval $[a,b]\subset I$ and
we define the variation $\gamma_\alpha$ as above
given by linear reparametrization of time.
Since $\gamma\mid_{[a,b]}$ is also a free
time minimizer, we must have
\[\frac{d}{d\alpha}A(\gamma_\alpha)
\mid_{\alpha=1}\,=\,-h(b-a)\,=\,0\]
which proves that $h=0$.
\end{proof}

\section{Existence of free time minimizers}

\subsection{Free time minimizers between two given configurations}

We know that for each configuration $x\in E^N$
and each $\tau>0$ there is a Tonelli minimizer
$\gamma\in\calC (x,x,\tau)$ defined in $[0,\tau]$.
It was proved (\cite{Mad} corollary 10) that we must have
$A(\gamma)=\phi(x,x,\tau)\leq \mu \tau^{1/3}$
for a constant $\mu>0$ which not depends on $x$.
Therefore we have $\phi(x,x)=0$. In other words,
given $x\in E^N$, we can leave $x$
and return to $x$ with small displacements,
in small times, and in such a way that
the action becomes arbitrarily small.
However, this does not happen when the
extremal configurations are different.
If $x\neq y$ and we try to minimize the
action from $x$ to $y$, we can see that
a curve defined on a short interval of
time has a too expensive action because
the average kinetic energy must be large.
On the other hand, also we will see that
once two configurations $x,y\in E^N$ are
fixed, the minimal action
$\phi(x,y,\tau)$ becomes arbitrarily
large for $\tau>0$ large enough.
These arguments enable us to prove the
following result.

\begin{theorem}
\label{FTMbetween2conf}
Given any two different configurations $x\neq y$
in $E^N$, there is $\tau>0$
and $\gamma\in\calC(x,y,\tau)$ such that
$A(\gamma)=\phi(x,y)$.
\end{theorem}

Of course, in case $x=y$ we can define the free time
minimizer as the constant curve on a trivial interval
of zero length, but this convention will be useless
for us. We will need the following lemma.

\begin{lemma}
\label{estimaciontaumin}
Let $d=\norm{x-y}$, $m_0=\min \set{m_1,\dots,m_N}$
and $\tau>0$. If $\gamma\in\calC(x,y,\tau)$ is such
that $A(\gamma)\leq A$, then
$2A\,\tau\geq m_0\,d^2$.
\end{lemma}

\begin{proof}
Let $x=(r_1,\dots,r_N)$ and
$y=(s_1,\dots,s_N)$.
Since $\norm{x-y}=\max\norm{r_i-s_i}_E$,
we can choose
$i_0\in\set{1,\dots,N}$ for which
$d=\norm{r_{i_0}-s_{i_0}}_E$.
Using now the inequality
(\ref{lowerboundaction}) of the precedent section,
we get $2A(\gamma)\,\tau\geq \abs{x-y}^2$. Since
$\abs{x-y}^2\geq m_{i_0}\norm{r_{i_0}-s_{i_0}}_E^2$
we conclude that $2A\,\tau\geq m_0\,d^2$.
\end{proof}

\begin{proof}[Proof of theorem \ref{FTMbetween2conf}]
Let $x\neq y$ be two given configurations.
A sequence of curves $\gamma_n\in \calC(x,y,\tau_n)$,
$n\geq 0$, will be called minimizing if
it satisfies $\phi(x,y)=\lim A(\gamma_n)$. Of course,
the existence of such sequences of curves follows from
the definition of $\phi(x,y)$.

As a first step of the proof
we will show that, given a sequence
$\gamma_n\in \calC(x,y,\tau_n)$
for which $A=\sup\set{A(\gamma_n)}<+\infty$,
there are positive
constants $0<T_0<T_1$ such that for all $n\geq 0$
we have $T_0\leq \tau_n \leq T_1$.
We observe first that by lemma \ref{estimaciontaumin},
we know that the lower bound $\tau_n\geq T_0= m_0\,d^2/2A$,
where $d=\norm{x-y}>0$ and $m_0=\min \set{m_1,\dots,m_N}$,
holds for all $n\geq 0$.
Moreover, if we fix any $n\geq 0$, and we restrict
the curve $\gamma_n$ to an interval $[0,t]$ with
$t\leq \tau_n$, once again by application of lemma
\ref{estimaciontaumin} we deduce that
\[d_n(t)=\norm{\gamma_n(t)-x}\leq (2A\,\tau_n/m_0)^{1/2}\]
hence we have
\begin{equation}\label{31}
\norm{\gamma_n(t)}\leq\norm{x}+(2A\,\tau_n/m_0)^{1/2}.
\end{equation}
Once we know the positions are bounded,
we get a lower bound for the
Newtonian potential throughout the curve
$\gamma_n$. Indeed, if $\norm{z}\leq K$ then
$U(z)\geq m_0^2/2K$. Thus we have
\[U(\gamma_n(t))\geq
\frac{m_0^2}{2(\norm{x}+(2A\,\tau_n/m_0)^{1/2})}\]
for all $t\in[0,\tau_n]$, and we conclude that
the inequality
\[A\geq A(\gamma_n)\geq
\frac{\tau_n\,m_0^2}{2(\norm{x}+(2A\,\tau_n/m_0)^{1/2})}\]
holds for all $n\geq 0$. But the right hand of the last
inequality is upper bounded if and only if the sequence
$\tau_n$ also is it. Thus we have proved the existence
of the positive constants $T_0<T_1$ as required.

The second step of the proof consists in applying the
first step to some minimizing sequence of curves,
which allow us to deduce the existence of a new
minimizing sequence of curves, but contained in a fixed
set $\calC(x,y,\tau)$.

More precisely, we start by choosing
a minimizing sequence $\gamma_n\in \calC(x,y,\tau_n)$,
$n\geq 0$. Since we know that
$0<T_0\leq \tau_n\leq T_1$ for all $n\geq 0$,
we can assume without loss of generality that
$\tau_n\to \tau>0$. Then we define
$\gamma^*_n:[0,\tau]\to E^N$
by $\gamma^*_n(t)=\gamma_n(\tau_n\,t/\tau)$,
and we can write
\[A(\gamma^*_n)=\frac{\tau_n}{\tau}\frac{1}{2}\,
\int_0^{\tau_n}\abs{\dot\gamma_n(t)}^2\,dt\,+\,
\frac{\tau}{\tau_n}\int_0^{\tau_n}U(\gamma_n(t))\,dt\]
from which we deduce that
$\lim A(\gamma^*_n)=\lim A(\gamma_n)=\phi(x,y)$.
Since each curve $\gamma^*_n$ is in $\calC(x,y,\tau)$,
we conclude that $\phi(x,y,\tau)\leq\phi(x,y)$. Thus
we must have $\phi(x,y,\tau)=\phi(x,y)$, which
reduces the proof to the application of the Tonelli's
theorem \ref{TonelliNBP}.
\end{proof}

\subsection{Homothetic free time minimizers}

Now we prove, as we announced in the introduction,
that every parabolic homothetic motion by a minimal
configuration $a\in E^N$ is a free time minimizer.
Recall that a minimal configuration is nothing but
a global minimum of the homogeneous function $I^{1/2}U$.

We start assuming that the minimal configuration
$a_0$ is also a normal configuration,
in the sense that $I(a_0)=1$.
Thus we have $U(a_0)=U_0$ where
$U_0=\min\set{U(x)\mid I(x)=1}$.
The corresponding parabolic homothetic ejection is
the curve $\gamma_0(t)=\mu_0\,t^{2/3}a_0$, where $\mu_0$ is
the only positive constant such that the curve $\gamma_0$
defines a motion for $t>0$. A simple
computation shows that the value of $\mu_0$
must be $(9\,U_0/2)^{1/3}$.
Note that $\gamma_0$ passes through $a_0$
in time $t_0=\mu_0^{-3/2}>0$.

\begin{lemma}
If $t_1>t_0$ then we have
$A(\gamma)\geq A(\gamma_0\mid_{[t_0,t_1]})$
for every $\gamma\in\calC(a_0,\gamma_0(t_1))$.
Moreover, the equality holds if and only if
$\gamma=\gamma_0\mid_{[t_0,t_1]}$
(modulo translation in time).
\end{lemma}

\begin{proof}
Let $\gamma$ be a curve in $\calC(a_0,\gamma_0(t_1))$.
Let us suppose that $\gamma $ is defined
in a given interval of time $[s_o,s_1]$.
Thus we have $\gamma(s_0)=\gamma_0(t_0)=a_0$
and $\gamma(s_1)=\gamma_0(t_1)$.
It is clear that there is a unique number
$s'\in [s_0,s_1)$ such that $I(\gamma(s'))=1$
and such that $I(\gamma(s))>1$ for every $s\in (s',s_1]$.
Thus we can define the curve $\gamma_1$ as
$\gamma\mid_{[s',s_1]}$, and obviously we have
$A(\gamma)\geq A(\gamma_1)$ with
equality holding if and only if $s'=s_0$.

Since $\gamma_1(s) \neq 0 $ we can now write $\gamma_1$
in polar coordinates $\gamma_1(s)=\rho_1(s)u_1(s)$,
where $\rho(s)\geq 1$ and $I(u_1(s))=1$ for all $s\in[s',s_1]$.
Let now $\gamma_2\in\calC(a_0,\gamma_0(t_1))$
be a second curve, that we define as
$\gamma_2(s)=\rho_1(s)a_0$ for  $s\in[s',s_1]$.
Using the expression of the Lagrangian action in
polar coordinates deduced in section (\ref{polaraction})
we conclude that $A(\gamma_2)\leq A(\gamma_1)$ and that the equality
holds if and only if $\gamma_2=\gamma_1$.

More precisely the action of $\gamma_2$ in polar coordinates is
\begin{equation}
A(\gamma_2)=\frac{1}{2}\int_{s'}^{s_1}\dot\rho_1(s)^2\,ds+
U_0\,\int_{s'}^{s_1}\rho_1(s)^{-1}\,ds.
\end{equation}
Note that this quantity is exactly the Lagrangian action
of $\rho_1(s)$ for the Kepler problem in the line with Lagrangian
\[L_{\kappa}(\rho,\dot\rho)= \frac{1}{2}\dot\rho\,^2+\frac{U_0}{\rho}\,.\]
On the other hand, we know that for this Keplerian Lagrangian
there is a free time minimizer curve from $\rho_1(s')=1$ and
$\rho_1(s_1)=\mu_0t_1^{2/3}$, and must have zero energy.
This assertion can be proved by direct computations,
or using the same arguments given in the proof of
theorem \ref{FTMbetween2conf}.
It is very easy to see that there is only one extremal
curve of the Lagrangian $L_{\kappa}$
(modulo translation of the time interval),
with zero energy and the required extremities,
namely $\rho(s)=\mu_0s^{2/3}$ for $s\in[t_0,t_1]$.

We conclude that $A(\gamma_2)\geq A(\gamma_0\mid_{[t_0,t_1]})$.
Moreover, the equality holds if and only if
$\rho_1(s)=\mu_0(t_0+s-s')^{2/3}$ for all $s\in [s',s_1]$,
in which case we must also have $s_1-s'=t_1-t_0$.

The above considerations finish the proof of the lemma,
since we have showed that
\[A(\gamma)\geq A(\gamma_1)\geq\ A(\gamma_2)
\geq A(\gamma_0\mid_{[t_0,t_1]})\]
and that we have equality if and only if
$\gamma(s)=\gamma_0(t_0+(s-s_0))$.
\end{proof}

\begin{proposition}
Let $a\in E^N$ be a minimal configuration, and $\mu>0$ such
that the curve defined for $t>0$ as $\gamma(t)=\mu\, t^{2/3}a$
is an homothetic (parabolic) motion.
The continuous extension of $\gamma$ to
$[0,+\infty)$ is a free time minimizer with
total collision at $t=0$.
\end{proposition}

\begin{proof}
In order to apply the previous lemma, we write $\gamma$
in the form $\gamma(t)=\mu_0\,t^{2/3}a_0$ with $I(a_0)=1$.
Therefore we know that, if $\gamma(t_0)=a_0$ and $t_1>t_0$,
then $\gamma\mid_{[t_0,t_1]}$ is a free time minimizer.

Let us fix $T>0$ and $\epsilon\in (0,T)$.
Taking $\lambda>0$ in such a way that $\epsilon=\lambda^{3/2}t_0$
and using corollary \ref{inv-rescalingFTM}, we can deduce that
$\gamma\mid_{[\epsilon,T]}$ is also a free time minimizer.
This means that
\[A(\gamma\mid_{[\epsilon,T]})=\phi(\gamma(\epsilon),\gamma(T))\,.\]
But
\[\lim_{\epsilon\to 0}A(\gamma\mid_{[\epsilon,T]})=
A(\gamma\mid_{[0,T]})\]
and $\phi$ is continuous, so we conclude that $\gamma\mid_{[0,T]}$
is a free time minimizer.
Since $T>0$ is arbitrary, the proof is complete.
\end{proof}

\subsection{Calibrating curves of weak KAM solutions}

As we said, thanks to the weak KAM theorem we know that
there are a lot of free time minimizers defined over unbounded
intervals. Moreover, this theory allows to establish that
for any configuration of bodies $x\in E^N$, there is at least one
free time minimizer $\gamma_x(t)$ defined for all time $t\geq 0$
and such that $\gamma_x(0)=x$ (see \cite{Mad}, proposition 15).

The interesting fact here, is that we have a lamination of
$E^N$ by such curves associated to each weak KAM solution
of the Hamilton-Jacobi equation. They are called calibrating
curves of the weak KAM solution. One of the main reasons for
studying the dynamics of these curves is precisely the link
with weak KAM theory. We hope that the results presented here
will be useful to characterize the set of these weak solutions,
either in the general case or for generic values ​​of the masses.

\section{Proof of theorems \ref{main} and \ref{nocompleteFTM}}

The following two results established by Pollard
in \cite{Poll} will be used in the proof of
theorem \ref{main}.
In both cases it is assumed that the center of mass
of the motion is fixed at $0\in E$.

\begin{theorem}[Pollard \cite{Poll}, theorem 1.2]
\label{Ilowerbound}
If the energy of a motion
$x:[t_0 ,+\infty)\to \Omega$ is zero, and
the center of mass
satisfies $G(x(t))=0$ for all $t\geq t_0$,
then there is a positive constant
$\alpha_0>0$ such that
$I(t)\geq \alpha_0 (t-t_0)^{4/3}$ for all
$t\geq t_0$.
\end{theorem}

\begin{theorem}[Pollard \cite{Poll}, theorem 5.1]
\label{PollardAlternative}
Let $x:[t_0 ,+\infty)\to \Omega$ be a motion
of zero energy such that the center of mass
satisfies $G(x(t))=0$ for all $t\geq t_0$.
Then either
\[U(t)\sim \alpha\,t^{-2/3} \;\;\;\textit{ and }
\;\;\;I(t)\sim (9/4)\,\alpha\, t^{4/3}\,,\]
for some positive constant $\alpha>0$, or
\[\lim r(t)\,t^{-2/3}=0 \;\;\textit{ and }
\;\;\lim I(t)\,t^{4/3}= +\infty\,,\]
where $r(t)=\min\set{\norm{r_{ij}(t)}\mid 1\leq i<j\leq N}$.
\end{theorem}

Recall now that the center of mass of
a given configuration $x=(r_1,\dots,r_N)\in E^N$
is barycenter of the weighted positions $r_i$.
We can define it as the linear map $G:E^N \to E$
given by $G(x)=M^{-1}\sum m_ir_i$, where
$M=\sum m_i$ is the total mass of the system.

The important property of the center of mass is
that he has an affine motion when it is computed
along every motion. This fundamental property
which correspond to the conservation of the linear moment,
allow us to reduce the study of a given motion
$x(t)$ to the corresponding {\em internal motion},
$y(t)=(r_1(t)-G(x(t)),\dots,r_N(t)-G(x(t))$
, since $y(t)$ must be
also a solution of Newton's equations. Therefore we can
say that every motion is the composition of an uniform
translation in space with a particular motion
contained in $\ker G$ which is nothing but the
orthogonal space of the diagonal $\Delta\subset E^N$
with respect to the mass inner product.

For the above reason it is usual in almost all the
literature on the subject, to assume that the motions
have the center of mass fixed at $0\in E$.

In order to prove our main results, we shall combine these
theorems of Pollard with the following lemma
recently proved by the second author, which exclude the
collinear case in theorem \ref{main},
since his proof uses Marchal's theorem.

\begin{lemma}[\cite{Mad2}, lemma 3]\label{fixedCM}
If $\dim E\geq 2$ and $x:[0,+\infty)\to E^N$ is a free time minimizer then the
center of mass $G(x(t))$ is constant.
\end{lemma}

However, we expect that this lemma remains true
even in case $\dim E=1$, as well as theorem \ref{main}.
This would be true for example if we could prove
the following conjecture.

\begin{conjecture}
If $\dim E=1$ and $x:[0,+\infty)\to E^N$
is a free time minimizer then there is a
finite set $\calT_x\subset [0,+\infty)$
such that $x(t)$ is a configuration with collisions
if and only if $t\in \calT_x$.
\end{conjecture}

The collinear case seems to be more approachable,
since we know that there are exactly $n!$ central
configurations and it can be proved that for generic
values of the masses only two (symmetric) of such
configurations are minimal.

We start the proof with the analysis of
the inertia of a free time minimizer.

\begin{proposition}\label{g-bounded}
Let $x:[t_0,+\infty)\to\Omega$ be a free time minimizer.
Then the function $g:[t_0,+\infty)\to\R$ defined by
\[g(t) = \dot I(t)\,I(t)^{-1/4}\]
is increasing and bounded.
\end{proposition}

\begin{proof}
Since $h=0$, we have that $T(t)=U(t)$
for all $t\geq t_0$, and the Lagrange-Jacobi
relation gives $\ddot I(t)=2U(t)=2T(t)$.
Thus the derivative of $g$ is
\begin{eqnarray*}
\dot g & = & \ddot I\,I^{-1/4}-
\frac{1}{4}\dot I^2I^{-5/4}\\
& = & \frac{1}{4}\left(
8\,IT-\dot I^2
\right)\,I^{-5/4}.
\end{eqnarray*}
But on the other hand, we know that
$2\,IT-\dot I^2\geq 0$. Therefore we conclude
that
\[\dot g \geq \frac{3}{2}\,UI^{-1/4}> 0.\]
Thus we have proved that the function $g$
is increasing. We must use now the minimization
property in order to prove that $g$ is bounded.

Fix $t>t_0$ and any normal configuration
$a\in E^N$, that is, such that $I(a)=1$.
We will compare the Lagrangian
action of the free time minimizer $x$ restricted
to the interval $[t_0,t]$ with the action of
the homothetic curve $\hat x:[t_0,t]\to E^N$
given by $\hat x(s)=\rho(s)a$ where
$\rho(s)=I(x(s))^{1/2}$. Here we will use the
polar notation $x=\rho\,u$ where $u(s)$
is the normalized configuration of $x(s)$.
Also we will write $\rho_0$ and $\rho_t$
for denote $\rho(t_0)$ and $\rho(t)$ respectively,
as well as $u_0$ and $u_t$ for $u(t_0)$ and
$u(t)$.

By the triangular inequality we have
\[A(x\mid_{[t_0,t]})=\phi(\rho_0u_0,
\rho_tu_t)\leq A(\hat x)+
\phi(\rho_0u_0,\rho_0 a)+
\phi(\rho_t a,\rho_tu_t).\]
Moreover, since $S=\set{u\in E^N\mid I(u)=1}$
is compact, using corollary \ref{phi-homogeneo}
we can write
\[A(x\mid_{[t_0,t]})\leq A(\hat x)+
\Lambda(\rho_0^{1/2}+\rho_t^{1/2})\,,\]
where $\Lambda=\max\set{\phi(x,y)\mid x,y\in S}$.

Using the formula for the
action in polar coordinates
\ref{polaraction-expression} we have that
\[A(x\mid_{[t_0,t]})=
\frac{1}{2}\int_{t_0}^t\dot\rho(s)^2\,ds+
\frac{1}{2}\int_{t_0}^t\rho(s)^2\dot u(s)^2\,ds+
\int_{t_0}^t\rho(s)^{-1}U(u(s))\,ds\,,\]
and that
\[A(\hat x)=
\frac{1}{2}\int_{t_0}^t\dot\rho(s)^2\,ds+
U(a)\int_{t_0}^t\rho(s)^{-1}\,ds\,.\]
We note that both expressions have the same
first term, that the second term in the
expression of $A(x\mid_{[t_0,t]})$
is positive, and that using the
Lagrange-Jacobi relation (which gives
$\ddot I=2\,U$ in this case) we can write
\[\int_{t_0}^t\rho(s)^{-1}U(u(s))\,ds=
\int_{t_0}^tU(x(s))\,ds=
\frac{1}{2}\int_{t_0}^t\ddot I(s)\,ds=
\frac{1}{2}(\dot I(t)-\dot I(t_0))\,.\]
Therefore, from the above considerations
and the previous inequality we deduce that
\[\frac{1}{2}(\dot I(t)-\dot I(t_0))\leq
U(a)\int_{t_0}^t\rho(s)^{-1}\,ds+
\Lambda(\rho_0^{1/2}+\rho_t^{1/2})\,.\]
But $\rho(s)=I(s)^{1/2}$, thus by theorem
\ref{Ilowerbound} we have
\[\int_{t_0}^t\rho(s)^{-1}\,ds\leq
\alpha_0^{-1/2}\int_{t_0}^t(s-t_0)^{-2/3}\,ds=
3\alpha_0^{-1/2}\,
(t-t_0)^{1/3}\,,\]
and we get the inequality
\[\dot I(t)\leq \alpha_1(t-t_0)^{1/3}+
\alpha_2\,I(t)^{1/4}+\alpha_3\]
for some positive constants $\alpha_1$,
$\alpha_2$ and $\alpha_3$. Finally,
dividing by $I(t)^{1/4}$ and
using again theorem \ref{Ilowerbound}
we get
\[g(t)=\dot I(t)\,I(t)^{-1/4}\leq
\alpha_1\alpha_0^{-1/4}+\alpha_2+
\alpha_3\alpha_0^{-1/4}(t-t_0)^{-1/3}\]
from which we conclude that the function
$g$ is bounded.
\end{proof}

We can now deduce the following two corollaries

\begin{corollary}\label{Iuperbound}
If $x:[t_0,+\infty)\to\Omega$ is a free time minimizer then
there is a constant $\beta>0$
such that $I(t)\leq \beta\, t^{4/3}$.
\end{corollary}

\begin{proof}
By the previous proposition we know that
there is a positive constant $\beta_0$
such that $\dot I\,I^{-1/4}<\beta_0$.
Integrating between $t_0$ and $t>t_0$ we
get
\[\frac{4}{3}\,(I(t)^{3/4}-I(t_0)^{3/4})
\leq \beta_0(t-t_0)\]
hence
\[I(t)\leq (\beta_1(t-t_0)+\beta_2)^{4/3}\]
for some positive constants $\beta_1$ and $\beta_2$.
Therefore $I(t)\,t^{-4/3}$ must be bounded.
\end{proof}

\begin{corollary}\label{UandI}
If $\dim E\geq 2$, and
$x:[t_0,+\infty)\to E^N$
is a free time minimizer, then
\[U(t)\sim \alpha\,t^{-2/3} \;\;\;\textit{ and }
\;\;\;I(t)\sim (9/4)\,\alpha\, t^{4/3}\,,\]
for some positive constant $\alpha>0$.
\end{corollary}

\begin{proof}
By lemma \ref{fixedCM} we know that $G(x(t))$
is constant. If we call $G$ this constant vector
of $E$,
$\delta_G=(G,\dots,G)\in E^N$
the configuration of total collision at $G$,
and we write $y(t)=x(t)-\delta_G$ for
the \emph{internal motion} of $x$,
then it is well known or easy to check that:
\begin{enumerate}
\item $y:[t_0,+\infty)\to E^N$ is also a
free time minimizer,
\item $G(y(t))=0$ for all $t\geq t_0$,
\item $I(x(t))=M\norm{G}^2+I(y(t))$ where
$M=m_1+\dots+m_N$ is the total mass, and
\item $U(x(t))=U(y(t))$ for all $t\geq t_0$.
\end{enumerate}
In particular, Marchal's theorem implies
that $y(t)\in\Omega$ for every $t>t_0$ and
we can apply corollary \ref{Iuperbound}
to the curve $y(t)$.
The proof follows then from Pollard's theorem
\ref{PollardAlternative}.
\end{proof}

\begin{proof}[Proof of theorem \ref{main}.]
Suppose that $x:[t_0,+\infty)\to E^N$ is a free time minimizer
and that $\dim E\geq 2$. By corollary \ref{UandI} we have
that $U(t)\to 0$. This implies that all mutual distances
$r_{ij}(t)$ tend to infinity. Moreover, since by lemma
\ref{zeroenergy} we know that the energy of the
motion is zero, we also have that $T(t)\to 0$, which is
equivalent to say that $\dot r_i(t)\to 0$ for all
$i=1,\dots,N$.
\end{proof}

\begin{proof}[Proof of theorem \ref{nocompleteFTM}.]
Suppose that $x:(-\infty,+\infty)\to E^N$
is a free time minimizer and that $\dim E\geq 2$.
Since $I(t)>0$ the normalized
configuration $u(t)=I(t)^{-1/2}x(t)$ is well defined
for all $t\in\R$. The set of normal configurations
$S=\set{x\in E^N\mid I(x)=1}$ is compact, therefore there
should be an increasing sequence of positive integers
$(n_k)_{k\geq 0}$ and normal configurations
$a,b\in S$ such that
$\lim u(-n_k)=a$ and $\lim u(n_k)=b$.
Note that application of corollary \ref{UandI}
we know that
there are positive constants $\alpha,\beta>0$
such that
$I(t)\sim \alpha^2\,t^{4/3}$ for $t\to -\infty$
and
$I(t)\sim \beta^2\,t^{4/3}$ for $t\to +\infty$.

To each $k\geq 0$ we will associate a free time
minimizer defined on the interval $[-1, 1]$ using
corollary \ref{inv-rescalingFTM} and the restriction
of $x$ to the interval $[-n_k,n_k]$. Thus
the sequence of free time minimizers is given
by
\[\gamma_k:[-1,1]\to E^N\;\;\;
\gamma_k(t)=n_k^{-2/3}x(n_k\,t)\,.\]
Hence we have
\[\lim \gamma_k(-1)=
\lim (n_k^{-2/3}I(-n_k)^{1/2}).
\lim u(-n_k)=\alpha\, a\,,\]
and
\[\lim \gamma_k(1)=
\lim (n_k^{-2/3}I(n_k)^{1/2}).
\lim u(n_k)=\beta\, b\,.\]
Since for each $k\geq 0$ the curve $\gamma_k$
is a free time minimizer, we have that
\[\lim A(\gamma_k)=
\lim\phi(\gamma_k(-1),\gamma_k(1))=
\phi(\alpha\,a,\beta\,b)\,.\]
Therefore we can apply proposition \ref{subsequence},
and we deduce that there is a subsequence of
$(\gamma_k)_{k\geq 0}$
which converges uniformly to some free time minimizer
$\gamma:[-1,1]\to E^N$. In particular, we must have
$A(\gamma)=\phi(\alpha\,a,\beta\,b)$ and $\gamma(0)=0$,
but this is impossible because it contradicts Marchal's
theorem.
\end{proof}

\textsl{Acknowledgments.}
The first author would like to thank Soledad Villar
for her infinite patience, support and for her useful advise.
The second author would like to express its gratitude to
Alain Chenciner, Albert Fathi, Christian Marchal,
Susanna Terracini and Andrea Venturelli for very helpful
discussions and comments on the subject.
Finally, we thank the referee for providing constructive
comments and help in improving the contents of this paper.

\end{document}